\documentclass[11pt]{amsart}
\usepackage{amsmath}
\usepackage{amssymb}
\usepackage{amsfonts}
\usepackage{mathrsfs}
\usepackage{amscd}
\usepackage{enumerate}
\usepackage[sans]{dsfont} 
\DeclareMathAlphabet{\mathpzc}{OT1}{pzc}{m}{it} 

\numberwithin{equation}{section}       
\numberwithin{figure}{section}       

\theoremstyle{plain}
\newtheorem{theo}{Theorem}
\newtheorem{prop}{Proposition}[section]

\newtheorem{lemm}[prop]{Lemma}

\theoremstyle{definition}

\theoremstyle{remark}
\newtheorem{rema}[prop]{Remark}

\newtheoremstyle{citing}
  {3pt}
  {3pt}
  {\itshape}
  {}
  {\bfseries}
  {.}
  {.5em}
  {\thmnote{#3}}

\theoremstyle{citing}
\newtheorem*{generic}{}

\newcommand{\partn}[1]{{\smallskip \noindent \textbf{#1.}}}
\newcommand{\R}{\mathbb{R}}
\newcommand{\Z}{\mathbb{Z}}

\newcommand{\cM}{\mathcal{M}}
\newcommand{\cZ}{\mathcal{Z}}
\newcommand{\sA}{\mathscr{A}}
\newcommand{\sB}{\mathscr{B}}
\newcommand{\hC}{\widehat{C}}
\newcommand{\hK}{\widehat{K}}
\newcommand{\hS}{\widehat{S}}
\newcommand{\hvarphi}{\widehat{\varphi}}
\newcommand{\tC}{\widetilde{C}}
\newcommand{\tM}{\widetilde{M}}

\newcommand{\tW}{\widetilde{W}}
\newcommand{\tnu}{\widetilde{\nu}}
\renewcommand{\=}{ : = }
\newcommand{\eps}{\varepsilon}
\newcommand{\ul}{\underline{l}}
\newcommand{\whf}{\widehat{f}}

\DeclareMathOperator{\dist}{dist}
\DeclareMathOperator{\Crit}{Crit} 

\begin{document}

\title[Equivalent characterizations of hyperbolic potentials]{Equivalent characterizations of hyperbolic H\"{o}lder potential for interval maps }

\author{Huaibin Li}
\address{Huaibin Li, Facultad de Matem{\'a}ticas, Pontificia Universidad Cat{\'o}lica de Chile, Avenida Vicu{\~n}a Mackenna~4860, Santiago, Chile}
\thanks{The author  was supported by the National Natural Science Foundation of China (Grant No.
11101124), and  FONDECYT grant 3110060, Chile.}

\email{matlihb@gmail.com}

\maketitle
\begin{abstract}
Consider a topologically exact $C^3$ interval map without non-flat critical points. Following the works we did in~\cite{LiRiv12two}, we give two equivalent characterizations of hyperbolic H\"{o}lder continuous potential in terms of the Lyapunov exponents and the measure-theoretic entropies  of equilibrium states for those potentials.
\end{abstract}
\section{Introduction}
The thermodynamic formalism of smooth dynamical systems was initiated  by Sinai, Ruelle, and Bowen~\cite{Bow75,Rue76e}.
For a uniformly hyperbolic diffeomorphism acting on a compact manifold of arbitrary dimension, they gave a complete description for H{\"o}lder continuous potentials. There have been several extensions of these results to one-dimensional maps, that go beyond the uniformly hyperbolic setting.
The lack of uniform hyperbolicity is usually compensated by an extra hypothesis on the potential.  For example, there is a wealth of results for a piecewise monotone interval map~$f : I \to I$ and a potential~$\varphi$ of bounded variation satisfying
$ \sup_I \varphi < P(f, \varphi),$
where~$P(f, \varphi)$ denotes the  pressure, see for example~\cite{BalKel90,DenKelUrb90,HofKel82b,Kel85}
and references therein, as well as Baladi's book~\cite[\S$3$]{Bal00b}.
Most results apply under the following weaker condition:
\begin{center}
\emph{For some integer~$n \ge 1$, the function~$S_n(\varphi) \= \sum_{j = 0}^{n - 1} \varphi \circ f^j$ satisfies \\
$\sup_{I} \frac{1}{n} S_n(\varphi) < P(f, \varphi).$}
\end{center}
In what follows, a potential~$\varphi$ satisfying this condition is said to be \emph{hyperbolic for~$f$}.

In this paper, our goal is to give  equivalent characterizations of hyperbolic H\"{o}lder continuous potentials for interval maps. In~\cite{InoRiv12} the authors gave  characterizations of hyperbolic H\"{o}lder continuous potentials
for rational maps. Here we want to extend this result to interval maps in order to refer in future.
In order to state our main result, we briefly recall some concepts from
thermodynamic formalism, see for example\cite{PrzUrb10}
for background.

Let~$(X, \dist)$ be a compact metric space and $T:X\to X$ a continuous map.
Denote by~$\cM(X)$ the space of Borel probability measures on~$X$ endowed with the weak* topology, and by~$\cM(X, T)$ the subspace of~$\cM(X)$ of those measures that are invariant by~$T$.
For each measure~$\mu$ in~$\cM(X, T)$, denote by~$h_{\mu}(T)$ the \emph{measure-theoretic entropy} of~$\mu.$
For a continuous function $\varphi: X \to \R$, denote by~$P(T,\varphi)$ the \emph{topological pressure of $T$ for the potential $\varphi$}, defined by
\begin{equation}
\label{e:variational principle}
P(T, \varphi)
\=
\sup\left\{h_\mu(T) + \int_X \varphi \ d\mu: \mu\in \cM(X, T)\right\}.
\end{equation}
A measure~$\mu$ in~$\cM(X, T)$ is called an \emph{equilibrium state of $T$ for the potential~$\varphi$}, if the supremum in~\eqref{e:variational principle} is attained at~$\mu$.

Given a compact interval~$I$ of $\R$, and a differentiable map~$f : I \to I$, a point of~$I$ is \emph{critical} if the derivative of~$f$ vanishes at it. Denote by~$\Crit(f)$ the set of critical points of~$f$.
A differentiable interval map~$f : I \to I$ is \emph{of class~$C^3$ with non-flat critical points}, if it has a finite number of critical points and if:
\begin{itemize}
\item
The map~$f$ is of class~$C^3$ outside $\Crit(f)$;
\item
For each critical point~$c$ of~$f$ there exists a number $\ell_c>1$ and diffeomorphisms~$\phi$ and~$\psi$ of~$\R$ of class~$C^3$, such that $\phi(c)=\psi(f(c))=0$, and such that on a neighborhood of~$c$ on~$I$, we have
$$ |\psi\circ f| = |\phi|^{\ell_c}. $$
\end{itemize}

Throughout the rest of this paper, fix a compact interval $I$ of $\R$ and let~$\sA$ denote  the collection of interval maps $f:I\to I$ of class~$C^3$ with non-flat critical points. For an interval map $f$ of $\sA$, denote by~$|\cdot|$ the distance on~$I$ induced by the norm distance on~$\R$. Besides, for a subset~$W$ of~$I$ we use~$|W|$ to denote the diameter of~$W$ with respect to~$|\cdot|$. For each measure~$\mu$ in~$\cM(I, f)$,  denote the Lyapunov exponent of $\mu$ by $$\chi_{\mu}(f)\=\int_{X} \ln |f'| d\mu.$$

In what follows, we say that $f$ is \emph{topologically exact} if for every open subset $U\subset I$ there is $n\geq 1$ such that $f^n(U)=I.$
The main result of this paper is following:
\begin{theo}
\label{t:real}
Let $f: I \to I$ be an interval map in $\sA$. If $f$ is  topologically exact.
Then for every H\"{o}lder continuous potential $\varphi: I\to \R,$ the following properties are equivalent:
\begin{enumerate}[1.]
\item The potential $\varphi$ is hyperbolic for $f$;
\item The measure-theoretic entropy of each equilibrium state of $f$ for the potential $\varphi$ is strictly positive.
\item The Lyapunov exponent of each equilibrium state of $f$ for the potential $\varphi$ is strictly positive.
\end{enumerate}
\end{theo}
\begin{rema}
\label{r:TCE}
The equivalence of properties $1$ and $2$ of Theorem~\ref{t:real} is part of Proposition 3.1 of~\cite{InoRiv12}, although we give the proof for the reader's convenience.  On the other hand, recall that a map~$f$ in~$\sA$ satisfies the \emph{Topological Collet-Eckmann Condition}, if there is a constant~$\chi > 0$ such that for every~$\nu$ in~$\cM(I, f)$ we have~$\int_I \ln |f'| \ d\nu \ge \chi$, see~\cite{Riv1204} for other equivalent formulations.
 Let $f$ be a topologically exact map in~$\sA$ that satisfies the Topological Collet-Eckmann condition, then every H\"{o}lder continuous potential is hyperbolic for $f,$ see also~\cite{LiRiv12two} for another proof.
\end{rema}

\subsection{Acknowledgments}
\label{ss:Acknowledgements}
The author would like to thank  Juan Rivera-Letelier for his useful discussions and helps.

\section{A reduction}
\label{s:a reduction}
In this section, we state first our main technical result as the ``Key Lemma'', whose proof occupies \S\S\ref{s:constructing a free IMFS}, \ref{s:proof of Key Lemma}, and then we derive our Theorem~\ref{t:real} from it.
In what follows, for each function $\varphi : I\to \R$ and each integer~$n \ge 1$, put
$$ S_n(\varphi)
\=
\varphi + \varphi \circ f + \cdots + \varphi \circ f^{n-1}. $$

\begin{generic}[Key Lemma]
Let~$f$ be a map in $\sA$ that is topologically exact and let $\nu$ be an invariant ergodic probability measure whose Lyapunov exponent is strictly positive.
Then for every H\"{o}lder continuous function $\varphi:I\to \R$,
there is a set of full measure of points~$x_0$ such that
$$ \limsup_{n\rightarrow \infty}\frac{1}{n}\log \sum_{y\in f^{-n}(x_0)}\exp \left( S_n(\varphi)(y) \right)
>
\int_I \varphi \ d\nu. $$
\end{generic}

The following lemma is useful for the proof Theorem~\ref{t:real}.
\begin{lemm}[Lemma~2.8, \cite{LiRiv12two}]
\label{l:tree pressure}
Let~$f$ be a map in~$\sA$ that is topologically exact, and let~$\varphi : I \to \R$ be a continuous function.
Then for every point~$x_0$ of~$I,$ we have
$$ P(f, \varphi)
\ge
\limsup_{n\rightarrow \infty}\frac{1}{n}\log \sum_{y\in f^{-n}(x_0)} \exp (S_n\varphi(y)). $$
\end{lemm}

\begin{proof}[Proof of Theorem~\ref{t:real} assuming the Key Lemma]
First, to prove the implication
$1\Rightarrow 2$, assume that~$\varphi: I\to \R$ is a hyperbolic potential for $f$ and let $\nu$ be an
equilibrium state of $f$ for the potential $\varphi$. Then there is an integer~$n\geq 1$ such that $\sup_{I}\frac{1}{n}S_n(\varphi)<P(f,\varphi)$ and $$P(f,\varphi)=h_\nu(f)+\int_{I} \varphi \ d\nu.$$ Since
$\nu$ is invariant measure for $f$, we have
$$\int_{I}S_n(\varphi) \ d \nu =\sum_{i=0}^{n-1}\int_{I} \varphi\circ f^i\ d\nu =n\int_{I}\varphi\ d\nu.$$  It follows that
\begin{align*}
h_\nu(f)&=P(f,\varphi)-\int_{I}\varphi \ d\nu=P(f,\varphi)-\frac{1}{n}\int_{I}S_n(\varphi)\ d \nu\\&\geq P(f,\varphi)-\sup_{I}\frac{1}{n}S_n(\varphi)>0.
\end{align*}
The implication
$2 \Rightarrow 3$ is a direct consequence of the following Ruelle's inequality: $\max 2\{\chi_\nu(f), 0\}\geq h_\nu(f)>0,$ see for example~\cite{PrzUrb10, Rue78}.

It remains to prove the implication
$3 \Rightarrow 1$. First, as the proof of~\cite[Main Theorem]{LiRiv12two} without changes, we have
\begin{equation}
\label{e:mean average versus invariant average}
\limsup_{n\rightarrow \infty} \left( \sup_{I}\frac{1}{n}S_n(\varphi) \right)
\le
\sup_{\nu\in \cM(I, f)} \int_I \varphi \ d\nu.
\end{equation}
To prove that~$\varphi$ is hyperbolic for~$f$, let~$\nu_0$ be an invariant probability measure maximizing the function~$\nu \mapsto \int_I \varphi \ d \nu$.
Then for almost every ergodic component~$\nu_0'$ of~$\nu_0$, we have~$\int_I \varphi \ d \nu_0' = \int_I \varphi \ d \nu_0$.
Thus, the Key Lemma applied to such a~$\nu_0'$, together with Lemma~\ref{l:tree pressure} implies
$$ P(f, \varphi)
>
\int_I \varphi \ d\nu_0'
=
\int_I \varphi \ d\nu_0
=
\sup_{\nu \in \cM(I, f)} \int_I \varphi \ d\nu. $$
Together with~\eqref{e:mean average versus invariant average}, this implies that~$\varphi$ is hyperbolic for~$f$ and completes the proof of  Theorem~\ref{t:real}.
\end{proof}

\section{Proof of the Key Lemma}

In this section, we construct first an ``Iterated Multivalued Function System", and then use it to prove our Key Lemma.

\subsection{Iterated Multivalued Function Systems}\label{s:constructing a free IMFS}
This subsection is devoted to the construction of an ``Iterated Multivalued Function System", which is the main ingredient in the proof of the Key Lemma.
It is stated as Proposition~\ref{p:constructing a free IMFS}, below.

Let~$f$ be a map in~$\sA$.
Given a compact and connected subset~$B_0$ of~$I$, a sequence multivalued functions~$(\phi_l)_{l = 1}^{+\infty}$ is an \emph{Iterated Multivalued Function System (IMFS) generated by~$f$}, if for every~$l$ there is an integer $m_l \ge 1$, and a pull-back~$W_l$ of~$B_0$ by~$f^{m_l}$ contained in~$B_0$, such that
$$ f^{m_l}(W_l) = B_0\,\,\,
\text{ and }\,\,\,
\phi_l = (f^{m_l}|_{W_l})^{-1}. $$
In this case, $(m_l)_{l = 1}^{+\infty}$ is the \emph{time sequence of $(\phi_l)_{l = 1}^{+\infty}$}, and \emph{$(\phi_l)_{l = 1}^{+\infty}$ is defined on~$B_0$}.
Note that for each subset~$A$ of~$B_0$ and each~$l$, the set~$\phi_l(A) \= f^{-m_l}(A)\cap W_l$ is non-empty.

Let~$(\phi_l)_{l = 1}^{+\infty}$ be an IMFS generated by~$f$ with time sequence~$(m_l)_{l = 1}^{+\infty}$, defined on a set~$B_0$.
For each integer~$n \ge 1$ put $\Sigma_n \= \{1, 2,\cdots \}^n $ and denote the space of all finite words in the alphabet~$\{1, 2, \ldots, \}$ by $\Sigma^* \= \bigcup_{n\ge 1}\Sigma_n$.
For every integer~$k \ge 1$ and~$\ul = l_1 \cdots l_k$ in~$\Sigma^*$, put
$$ | \ul |=k, \,\,
m_{\ul} = m_{l_1}+m_{l_2}+\cdots +m_{l_k}\,\,
\text{ and }\,\,
\phi_{\ul} = \phi_{l_1}\circ\cdots\circ\phi_{l_k}. $$
Note that for every~$x_0$ in~$B_0$, and every pair of distinct words~$\ul$ and~$\ul'$ in~$\Sigma^*$ satisfying~$m_{\ul} = m_{\ul'}$, we have the following property:
\begin{center}
(*) \qquad
If the sets~$\phi_{\ul}(x_0)$ and~$\phi_{\ul'}(x_0)$ intersect, then they coincide.
\end{center}
The IMFS $(\phi_l)_{l = 1}^{+\infty}$ is \emph{free}, if there is~$x_0$ in~$B_0$ such that for every pair of distinct words~$\ul$ and~$\ul'$ in~$\Sigma^*$ such that~$m_{\ul}=m_{\ul'}$, the sets~$\phi_{\ul}(x_0)$ and~$\phi_{\ul'}(x_0)$ are disjoint.

\begin{prop}
\label{p:constructing a free IMFS}
Let~$f$ be an interval map in~$\sA$ that is topologically exact.
Let~$\varphi: I\to \R$ be  H{\"o}lder continuous, $t\geq 0$ and put $\psi_t\=\varphi-t\ln|f'|.$ Let~$\nu$ be an ergodic invariant probability measure that is not supported on a periodic orbit and that has strictly positive Lyapunov exponent.
Then there exists a subset~$X$ of~$I$ of full measure with respect to~$\nu$, such that for every point~$x_0$ in~$X$ the following property holds: There exist~$D$ in $(0, +\infty)$, a compact and connected subset~$B_0$ of~$I$ containing~$x_0$, and a free IMFS~$(\phi_l)_{l=1}^{+\infty}$ generated by~$f$ with time sequence~$(m_l)_{l=1}^{+\infty}$, such that~$(\phi_l)_{l=1}^{+\infty}$ is defined on~$B_0$, and such that for every~$l$ and every~$y$ in~$\phi_l(B_0)$ we have
\begin{equation}\label{e:fre}
S_{m_l}(\psi_t)(y)\ge
m_{l}\int\psi_t \ d \nu - D.
\end{equation}
\end{prop}
The idea of the proof of this proposition is similar to the proof of~\cite[Proposition~3.1]{LiRiv12two}. Its proof, depending on several lemmas, is given at the end of this subsection.

We proceed first to recall the natural extension of~$f$.
Let $\Z_{-}$ denote the set of all non-positive integers and endow
$$ \cZ
\=
\left\{ (z_m)_{m\in \Z_{-}} \in I^{\Z_{-}} : \text{ for every } m \in \Z_-, f(z_{m - 1})=z_m \right\} $$
with the product topology.
Define $T: \cZ\rightarrow \cZ$ by
$$ T \left( (\cdots,z_{-2},z_{-1},z_0) \right)
=
(\cdots,z_{-2},z_{-1},z_0, f(z_0)) $$
and $\pi: \cZ\rightarrow I$ by $\pi((z_m)_{m\in \Z_{-}}) = z_0$.
Note that~$T$ is a bijection, $T^{-1}$ is measurable, $\pi$ is continuous and onto, and~$\pi\circ T= f\circ \pi$.
If~$\nu$ is a Borel probability measure~on~$I$ that is invariant and ergodic for~$f$, then there exists a unique Borel probability measure~$\tnu$ on~$\cZ$ that is invariant and ergodic for~$T$, and that satisfies~$\pi_* \tnu=\nu$, see for example~\cite[\S$2.7$]{PrzUrb10}. We call~$(\cZ,T, \tnu)$ the \emph{natural extension of $(I, f, \nu)$}.

The following is a well-known consequence of the pointwise ergodic theorem, see for example~\cite[Lemma~$1.3$]{PrzRivSmi04} for a proof.
\begin{lemm}\label{lem:ergodic}
Let $(\cZ, \sB, \tnu)$ be a probability space, and let $T : \cZ \rightarrow \cZ$ be an ergodic measure preserving transformation.
Then for each function $\phi: \cZ \rightarrow \R$ that is integrable with respect to~$\tnu$, there exists a subset~$Z$ of~$\cZ$ such that~$\tnu(Z) = 1$, and such that for every~$\underline{z}$ in~$Z$ we have
$$ \limsup_{n\rightarrow \infty}\sum_{i=0}^{n-1} \left(\phi(T^i(\underline{z})) - \int_{\cZ} \phi \ d\tnu \right)
\ge
0. $$
\end{lemm}

We need the following lemma which is a
version of Ledrappier's unstable manifold theorem, see~\cite{Dob0801} for the proof.

\begin{lemm}[Theorem~16, \cite{Dob0801}]\label{l:unstable}
Let $f$ be an interval map in $\sA.$ Suppose $\nu$ in $\cM(I, f)$ has strictly positive finite Lyapunov exponent.
Denote by $(\cZ, T, \tnu)$ the natural extension of $(I, f, \nu).$
Then there exists a measurable function $\alpha$ on $\cZ$ such that $0 < \alpha< 1/2$ almost everywhere with respect to~$\tnu$, and such that for $\tnu$-almost every point $y$ in $\cZ$  there exists a set $V_y$ contained in $\cZ$ with the
following properties:
\begin{enumerate}[1.]
\item $y$ is in $V_y$ and $\pi(V_y) = B(\pi(y), \alpha(y)).$
\item  For each integer $n \geq 0$, $f^n: \pi(T^{-n}(V_y)) \to \pi(V_y)$ is diffeomorphic.
\item For each $y'$ in $V_y$,
$$\sum_{i=0}^{+\infty}\left|\log |Df(\pi( T^{-i}(y')))|-\log |Df(\pi (T^{-i}(y')))| \right| < \log 2.$$
\item For each $\eta > 0$ there is a measurable function $\theta$ on $\cZ$ with~$0 < \theta < +\infty$
almost everywhere with respect to $\tnu$, and  such that
$$\frac{1}{\theta(y)}\exp(n(\chi_\nu-\eta))\leq |Df^n(\pi(F^{-n}(y)))|\leq \theta(y)\exp(n(\chi_\nu+\eta)).$$
In particular,
$$|\pi (T^{-n}(V_y))|\leq  2\theta(y)\exp(-n(\chi_\nu-\eta)).$$
\end{enumerate}
\end{lemm}

\begin{lemm}
\label{l:maximizing branches}
Let~$f:I\to I$ be a map in~$\sA$, ~$\nu$  an invariant ergodic probability measure with strictly positive Lyapunov exponent, $\varphi:I \rightarrow \R$ a H\"{o}lder continuous, and $t$ in $\R$.
Then there exists a subset~$X'$ of~$I$ of full measure with respect to~$\nu$, such that the following holds.
For every point~$x$ of~$X'$ there exist $\rho_x>0$, ~$D' > 0$ and a strictly increasing sequence of positive integers~$(n_l)_{l = 1}^{+\infty}$ such that for every~$l\geq 1$ we can choose a point~$x_l$ in~$f^{-n_l}(x)$ and a connected component $W_l$ of $f^{-n_l}(B(x, \rho_x))$ containing $x_l$ so that:
\begin{enumerate}[1.]
\item
$x_{l + 1}$ is in~$f^{-(n_{l + 1} - n_l)}(x_l)$.
\item Let $\psi_t\=\varphi-t\ln|f'|,$ then for every point~$y$ in $W_l$,
$$S_{n_l}(\psi_t)(y)\ge n_l\int_I \psi_t \ d\nu -D'.$$
\item $\lim_{l\to +\infty}|W_l|=0.$
\end{enumerate}
\end{lemm}

\begin{proof}
Let~$(\cZ, T, \tnu)$ be the natural extension of $(I, f, \nu),$ and note that~$\tnu$ is also ergodic with respect to~$T^{-1}$.
Applying Lemma~\ref{lem:ergodic} for~$T^{-1}$ to the integrable function~$\phi =\psi_t\circ \pi= (\varphi-t\ln|f'|) \circ \pi$, we obtain that there exists a subset~$Z$ of~$\cZ$ of full measure with respect to~$\tnu$, such that for every point~$(z_m)_{m \in \Z_{-}}$ in~$Z$ we have
\begin{equation}\label{e:lim}
\limsup_{n\rightarrow \infty}\sum_{i=0}^{n-1}\left(\psi_t\circ\pi \left( T^{-i} \left( (z_m)_{m \in \Z_{-}} \right) \right) - \int_{\cZ} \psi_t\circ\pi \ d\tnu\right)
\ge
0.
\end{equation}
Taking a subset of $Z$ of full measure with respect to $\tnu$ if necessary, by Lemma~\ref{l:unstable} we can assume that there is function~$\alpha: Z\to (0,1/2)$ such that~$Z$ and~$\alpha$ satisfy the assertions   of Lemma~\ref{l:unstable}. Since the set $X' \= \pi(Z)$ satisfies
$ \nu(X') = \tnu \left( \pi^{-1}(\pi (Z)) \right) \ge \tnu(Z) = 1, $
we have~$\nu(X') = 1$.

It remains to verify that~$X'$ satisfies the desired properties. Fix a point~$x$ in $X'$  and choose a point~$(z_m)_{m \in \Z_{-}}$ of~$Z$ such that~$\pi \left( (z_m)_{m\in \Z_{-}} \right) = x$. Let~$V_{(z_m)_{m \in \Z_{-}}}$ be given by Lemma~\ref{l:unstable} for the point $(z_m)_{m\in \Z_{-}}$, and put $\rho_x\=\alpha((z_m)_{m\in \Z_{-}}).$ Moreover, for each integer~$j \ge 1$ put
$$ y_j
\=
\pi \left( T^{-j} \left( (z_m)_{m\in \Z_{-}} \right) \right)
=
z_{j}
\in
f^{-j}(x),$$ and  $U_j\=\pi\left(F^{-j}(V_{(z_m)_{m\in \Z_{-}}})\right).$ By parts $1$ and $2$ of Lemma~\ref{l:unstable} we  know that for every integer $j\geq 1,$ $U_j$ is the connected component of $f^{-j}(B(x,\rho_x))$ containing $y_j$ and $f^j: U_j\to B(x,\rho_x)$ is diffeomorphic. On the other hand, by parts~$3$ and $4$ of Lemma~\ref{l:unstable} there exist $C'>0$ and $\lambda>1$ such that for every $n\geq 1$  we have $|U_n|\leq C'\lambda^{-n}$ and for every pair of points $x,y$ in $U_n$
\begin{equation}\label{e:d}
\frac{1}{2}\leq \frac{|(f^n)'(x)|}{|(f^n)'(y)|}\leq 2.
 \end{equation}
Since $\varphi$ is H\"{o}lder continuous, we have that there is $\tC>1$ such that for every $n\geq 1$ and  every pair of points $x,y$ in $U_n$
\begin{equation}\label{e:h}
|S_n(\varphi)(x)-S_n(\varphi)(y)|\leq \tC.
\end{equation}
Fix~$D'' > 0$.
Then by~(\ref{e:lim}) there is a strictly increasing sequence of positive integers $(n_l)_{l = 1}^{+\infty}$ such that for every integer~$l\geq 1,$ we have
\begin{equation}
\label{e:erg}
\sum_{i=0}^{n_l - 1} \psi_t \circ \pi \left( T^{-i} \left( (z_m)_{m \in \Z_{-}} \right) \right)
\ge
n_l\int_{\cZ} \psi_t \circ \pi\ d\tnu - D''=n_l\int_{I} \psi_t \ d\nu - D''.
\end{equation}
Therefore, if for each integer~$l \ge 1$ put $x_l\=y_{n_l}$ and $W_l\=U_{n_l}$,
then  parts~$1$ and~$3$ are  direct consequences of the definitions, and part~$2$ follows from~(\ref{e:d}), ~(\ref{e:h}) and~\eqref{e:erg} with $D'=D''+\tC+t\ln 2 $. The proof is completed.
\end{proof}

\begin{lemm}
\label{l:one-sided covering}
Let~$f$ be an interval map in~$\sA$ that is topologically exact, and~$x_0$  a interior point of~$I$.
Then for every open interval~$U\subset I$, and every sufficiently large integer~$n \ge 1$, there exist two distinct points~$y_1,y_2$ of~$f^{-n}(x_0)$ in~$U$ such that the following hold.
\begin{enumerate}[1.]
\item For every $\eps>0$ both of sets $f^n(B(y_1,\varepsilon))$ and $f^n(B(y_2,\varepsilon))$ intersect $(x_0, +\infty).$
\item For every $\eps>0$ both of sets $f^n(B(y_1,\varepsilon))$ and $f^n(B(y_2,\varepsilon))$ intersect $(-\infty, x_0)$.
\end{enumerate}
\end{lemm}
\begin{proof}
We only give the proof of part~$1$, and the proof can be applied to part~$2$ without changes. Let $U_1$ and $U_2$ be two disjoint open subintervals  of $U.$ Using that~$f$ is topologically  exact, we know that there is an integer $N\geq 1$ such that for every $n\geq N$ we have $f^n(U_1)=I$ and $f^n(U_2)=I$.
Fix $n\geq N$.
Note that the set~$f^{-n}(x_0)$ is finite, and there are points~$z_1\in U_1$ and $z_2\in U_2$  such that~$f^n(z_1)$ and $f^n(z_2)$ are in $(x_0, +\infty)$.
For each $i=1,2,$ let~$y_i$ be a point of $f^{-n}(x_0)$ in~$U_i$ such that for every $y'$ of $f^{-n}(x_0)$ in~$U_i$ we have $|y_i-z_i|\leq |y'-z_i|.$

Now let us prove the lemma holds for such $y_1$ and~$y_2.$ Obviously, $y_1$ and~$y_2$ are distinct. To prove that for every $\eps>0$  the set $f^n(B(y_1,\varepsilon))$ intersects $(x_0, +\infty)$, by contradiction, there is $\eps_0\in (0, |y_1-z_1|)$ such that~$f^n(B(y_1,\eps_0))$ is contained in $(-\infty, x_0]$.
It follows that there is a point~$z'$ of $B(y_1, \eps_0)\cap U_1$ such that~$f^n(z')$ is in $(-\infty, x_0)$ and $|z'-z_1|<|y_1-z_1|$.
Since $f^n$ is continuous on~$U_1$ it follows that there is $y''$ in $U_1$ such that $|y''-z_1|<|y_1-z_1|$ and $f^n(y'')=x_0.$ This is a contradiction with our choice of $y.$ Using the same method,  we can prove for every $\eps>0$  the set $f^n(B(y_2,\varepsilon))$ intersects $(x_0, +\infty)$. The lemma is proved.
\end{proof}

\begin{lemm}[Lemma~3.2, \cite{LiRiv12two}]
\label{l:almost properness}
For each interval map~$f : I \to I$ in~$\sA$ there is~$\varepsilon > 0$ such that the following property holds.
Let~$J_0$ be an interval contained in~$I$ satisfying~$|J_0| \le \varepsilon$, let~$n \ge 1$ be an integer, and let~$J$ be a pull-back of~$J_0$ by~$f^n$, such that for each~$j$ in~$\{1, \ldots, n \}$ the pull-back of~$J_0$ by~$f^j$ containing~$f^{n - j}(J)$ has length bounded from above by~$\varepsilon$.
If in addition the closure of~$J$ is contained in the interior of~$I$, then~$f^n(\partial J) \subset \partial J_0$.
\end{lemm}

\begin{lemm}[Lemma~A.2, \cite{Riv1206}]
\label{l:small pull}
Let $f: I \to I$ be an interval map in~$\sA$ that is  topologically exact.
Then for every $\kappa> 0$ there is $\delta> 0$ such that for every $x$ in $I$, every integer $n\geq 1$, and every pull-back $W$ of $B(x, \delta)$ by $f^n$, we have $|W| < \kappa.$
\end{lemm}

\begin{proof}[Proof of Proposition~\ref{p:constructing a free IMFS}]
Let~$\varepsilon > 0$ be the constant given by Lemma~\ref{l:almost properness} and let $\delta>0$ be  the constant given by Lemma~\ref{l:small pull} for $\kappa=\eps$. Let~$X'$ be the subset of~$I$ given by Lemma~\ref{l:maximizing branches}, and let~$X$ be the complement in~$X'$ of the set of periodic points of~$f$.
Since~$\nu$ is ergodic and it is not supported on a periodic orbit, the set~$X$ has full measure for~$\nu$.
Fix a point~$x_0$ of~$X$ that is not an endpoint of~$I$.

In part~$1$ below we define the IMFS, and in part~$2$ we show it is free and that it satisfies~\eqref{e:fre}.

\partn{1}
Let~$\rho_{x_0}$, $D'$, $(n_l)_{l = 1}^{+\infty}$, $(x_l)_{l = 1}^{+\infty}$ and~$(W_l)_{l = 1}^{+\infty}$ be given by Lemma~\ref{l:maximizing branches} with~$x = x_0$. Fix~$\rho$ in~$\left(0, \min\{\delta, \rho_{x_0},  \dist(x_0, \partial I)\} \right)$.
Taking a subsequence if necessary, assume~$(x_l)_{l = 1}^{\infty}$ converges to a point~$w_0$.
Since~$f$ is topologically exact, there exist an integer~$M \ge 1$ and distinct points~$y_0$, and~$y_1$ of~$(x_0 - \rho, x_0)$ such that $f^M(y_0) = f^M(y_1) = w_0$.
Let~$\rho'>0$ be such that the pull-backs~$U_0$ and~$U_1$ of~$B(w_0,\rho')$ by~$f^M$ containing~$y_0$ and~$y_1$, respectively, are disjoint and contained in~$B(x_0, \rho)$.
Moreover, by Lemma~\ref{l:one-sided covering} we can choose~$M$, $y_0$, and~$y_1$ so that in addition $U_0$, $U_1 \subset [x_0 - \rho, x_0]$, and so that there are infinitely many~$l$ for which~$x_l$ is contained in~$f^M(U_0)$ and in~$f^M(U_1)$.

Using that $\lim_{l\to +\infty}|W_l|=0$ and taking a subsequence if necessary, assume that for every~$l$ we have~$n_{l + 1} - n_l \ge M$, $|W_l|<\eps,$ that the point~$x_l$ is contained in~$f^M(U_0)$ and in~$f^M(U_1)$, and that the pull-back~$W_l$ of~$\overline{B(x_0,\rho)}$ by~$f^{n_l}$ containing~$x_l$ is contained in~$B(w_0,\rho')$.
Interchanging~$y_0$ and~$y_1$, and taking a subsequence if necessary, we can also assume that for every~$l$ the point~$f^{n_{l + 1} - n_l - M}(x_{l + 1})$ is not in~$U_0$.
For each~$l$ choose a pull-back~$W_l'$ of~$W_l$ by~$f^M$ that contains a point~$x_l'$ of~$f^{-M}(x_l)$ and that is contained in~$U_0$.

Note that~$W_l'$ is contained in~$U_0 \subset [x_0 - \rho, x_0]$, so the closure of~$W_l'$ is contained in the interior of~$I$. By Lemma~\ref{l:small pull} and the choice of $\rho$, we know that  for every $i$ in $\{0,1, \cdots, n-1\}$ the length of $f^i(W_l')$ is less than $\eps.$
So by Lemma~\ref{l:almost properness}  the set~$f^{n_l + M}(\partial W_l')$ is contained in~$\partial B(x_0, \rho)$.
Thus, for each~$l$ the set~$f^{n_l + M}(W_l')$ contains either~$[x_0 - \rho, x_0]$ or~$[x_0, x_0 + \rho]$.
Suppose first there are infinitely many~$l$ such that~$f^{n_l + M}(W_l')$ contains~$[x_0 - \rho, x_0]$.
Taking a subsequence if necessary, assume this holds for every~$l$.
Then for every~$l$ there is a pull-back~$W_l''$ of~$[x_0 - \rho, x_0]$ by~$f^{n_l + M}$ that is contained in~$W_l'$ and such that~$f^{n_l + M}(W_l'') = [x_0 - \rho, x_0]$.
In this case we put
$$ B_0 \= [x_0 - \rho, x_0],
M' \= M,
\text{ and }
U_0' \= U_0, $$
and note that~$W_l'' \subset W_l' \subset U_0' \subset [x_0 - \rho, x_0] = B_0$.
It remains to consider the case where for each~$l$, outside finitely many exceptions, the set~$f^{n_l + M}(W_l')$ contains~$[x_0, x_0 + \rho]$, but it does not contain~$[x_0 - \rho, x_0]$.
Taking a subsequence if necessary, assume this holds for every~$l$.
Since~$f$ is topologically  exact,
by Lemma~\ref{l:one-sided covering} there is an integer~$\tM \ge 1$ and a pull-back~$U_0'$ of~$U_0$ by~$f^{\tM}$ that is contained in~$(x_0, x_0 + \rho)$, and such that for infinitely many~$l$ the point~$x_l'$ is contained in~$f^{\tM}(U_0')$.
Taking a subsequence if necessary, assume that for every~$l$ we have~$n_{l + 1} - n_l \ge M + \tM$, and that the point~$x_l'$ is contained in~$f^{\tM}(U_0')$.
Since for each~$l$ the point~$f^{n_{l + 1} - n_l - M}(x_{l + 1})$ is not in~$U_0$, it follows that the point~$f^{n_{l + 1} - n_l - M - \tM}(x_{l + 1})$ is not in~$U_0'$.
For each~$l$ choose a pull-back~$\tW_l'$ of~$W_l'$ by~$f^{\tM}$ contained in~$U_0'$ and that contains a point of~$f^{- \tM}(x_l')$.
By Lemmas~\ref{l:almost properness} and~\ref{l:small pull} again, the set~$f^{n_l + M + \tM}(\partial \tW_l')$ is contained in~$\partial B(x_0, \rho)$.
Since the set~$f^{n_l + M + \tM}(\tW_l')$ is contained in~$f^{n_l + M}(W_l')$ and this last set does not contain~$[x_0 - \rho, x_0]$, we conclude that~$f^{n_l + M + \tM}$ maps both endpoints of~$\tW_l'$ to~$x = x_0 + \rho$.
Since by construction~$f^{n_l + M + \tM}(\tW_l')$ contains~$x = x_0$, we conclude that~$f^{n_l + M + \tM}(\tW_l')$ contains~$[x_0, x_0 + \rho]$.
So there is a pull-back~$W_l''$ of~$[x_0, x_0 + \rho]$ by~$f^{n_l + M + \tM}$ that is contained in~$\tW_l'$, and such that~$f^{n_l + M + \tM}(W_l'') = [x_0, x_0 + \rho]$.
Note that~$W_l'' \subset \tW_l' \subset U_0' \subset (x_0, x_0 + \rho)$.
In this case we put~$B_0 \= [x_0, x_0 + \rho]$, and~$M' \= M + \tM$.

Now for each integer~$l\geq 1$ we put
$$ \phi_l \= \left( f^{n_l + M'}|_{W_l''} \right)^{-1}. $$
Then, $(\phi_l)_{l = 1}^{+\infty}$ is an IMFS generated by~$f$ with time sequence~$(m_l)_{l=0}^{+\infty}\=(n_l + M')_{l=0}^{+\infty}$ that is defined on $B_0$.
Moreover, for each integer~$l\geq 1$ we have
$$ n_{l + 1} - n_l \ge M',
W_l'' \subset U_0',
\text{ and }
f^{n_{l + 1} - n_l - M'}(x_{l + 1}) \not \in U_0'. $$

\partn{2}
To prove that the IMFS~$(\phi_l)_{l = 1}^{+\infty}$ is free, let~$k \ge 1$ and $k'\geq 1$ be integers and let
$$ \ul\=l_1 l_2\cdots l_k
\text{ and }
\ul'\=l'_1 l'_2\cdots l'_{k'} $$
be different words in~$\Sigma^*$ such that~$m_{\ul} = m_{\ul'}$.
Assume without loss of generality that $l_{k'}' \ge l_{k} + 1$.
Note that the set
$$ f^{m_{\ul}-m_{l_k}}(\phi_{\ul}(x_0))
=
\phi_{l_k}(x_0). $$
is contained~$W_{l_k}''$, and therefore in~$U_0'$.
On the other hand, we have
$$ m_{l'_{k'}} - m_{l_k}
=
n_{l'_{k'}} - n_{l_k}
\ge
n_{l_k + 1} - n_{l_k}
\ge
M' $$
and therefore the set
\begin{multline*}
f^{m_{\ul} - m_{l_k}}(\phi_{\ul'}(x_0))
=
f^{m_{\ul'} - m_{l_k}}(\phi_{\ul'}(x_0))
=
f^{m_{l'_{k'}}-m_{l_k}}(\phi_{l'_{k'}}(x_0))
\\ =
f^{m_{l'_{k'}}-m_{l_k} - M'} \left( \left(f^{n_{l_{k'}'}}|_{W_{l_{k'}'}} \right)^{-1}(x_0) \right)
\end{multline*}
contains the point
$$ f^{m_{l_{k'}'} - m_{l_k} - M'} (x_{l_{k'}'})
=
f^{n_{l_{k'}'} - n_{l_k} - M'} (x_{l_{k'}'})
=
f^{n_{l_k + 1} - n_{l_k} - M'} (x_{l_k + 1}). $$
By construction this point is not in~$U_0'$, so we conclude that the sets
$$ f^{m_{\ul} - m_{l_k}}(\phi_{\ul}(x_0))
\text{ and }
f^{m_{\ul} - m_{l_k}}(\phi_{\ul'}(x_0)) $$
are different.
This implies that the sets~$\phi_{\ul}(x_0)$ and~$\phi_{\ul'}(x_0)$ are different, and by property~$(*)$ stated above the statement of the proposition, that they are disjoint.
This completes the proof that the IMFS $(\phi_l)_{l = 1}^{\infty}$ is free.

Finally, let us check inequality~\eqref{e:fre} in the statement of the proposition. Let $C_1\=\inf_I\varphi -t \sup_{I}\ln|f'|.$ Note that $t\geq 0$ and $\sup_I\ln|f'|>-\infty,$ then $-C_1<+\infty$ and $\inf_I \psi_t=\inf_I (\varphi-t\ln|f'|)\geq C_1$.
Recall that for every integer~$l\geq 1$ and~$y$ in~$\phi_l(B_0)$, the point~$f^{M'}(y)$ is in~$W_l$.
Thus, by
part~$2$ of Lemma~\ref{l:maximizing branches} we have
\begin{multline*}
S_{m_l}(\psi_t)(y)
 =
S_{n_l}(\psi_t)(f^{M'}(y)) + S_{M'}(\psi_t)(y)
\ge n_l\int_I \psi_t \ d \nu - D' +S_{M'}(\psi_t)(y)
\\  \ge
m_l\int_I \psi_t \ d \nu -D' -M'\int_{I}\psi_t \ d \nu+M' \inf_{I} \psi_t \\\geq m_l\int \psi_t \ d \nu -\left(D'+M'\left(\int_I \psi_t \ d \nu-C_1\right)\right).
\end{multline*}
This proves~\eqref{e:fre} with $D = \max\{0, D'+M'\left(\int_I \psi_t \ d \nu-C_1\right)\}<+\infty$, and completes the proof of the proposition.
\end{proof}

\subsection{Proof of the Key Lemma}
\label{s:proof of Key Lemma}
In this subsection we complete the proof of the Key Lemma.
The case where the measure~$\nu$ is supported on a periodic orbit is different.
The proof of the Key Lemma is completed after the following lemma. This lemma is an adaptation of~\cite[Lemma~4.1]{LiRiv12two}, and the proof of the Key Lemma is the same as that of~\cite[Key Lemma]{LiRiv12two}. We provide those proofs again for the reader's convenience.

Recall that for a differentiable map~$f : I \to I$, a periodic point~$p$ of~$f$ of period~$n$ is \emph{hyperbolic repelling}, if~$|D f^n(p)| > 1$.
\begin{lemm}
\label{l:periodic case}
Let~$f$ be an interval map in~$\sA$ that is topologically exact.
Then for every H{\"o}lder continuous function $\varphi: I\to \R$ and every hyperbolic repelling periodic point~$x_0$ of~$f$ of period~$N$, we have
\begin{equation}
\label{e:periodic gap pressure}
\limsup_{n\rightarrow \infty}\frac{1}{n}\log \sum_{y\in f^{-n}(x_0)} \exp (S_n(\varphi)(y))
>
\frac{1}{N} S_N(\varphi)(x_0).
\end{equation}
\end{lemm}

\begin{proof}
The proof is divided into~$2$ parts.
In part~$1$ we construct an induced map and in part~$2$ we show an inequality analogous to~\eqref{e:periodic gap pressure} for the induced map, from which~\eqref{e:periodic gap pressure} follows as a direct consequence.

\partn{1}
Since~$x_0$ is hyperbolic repelling, there is $\rho>0$ such that there is a local inverse~$\phi$ of~$f^{2N}$ defined on~$B(x_0,\rho)$ and fixing~$x_0$.
Reducing~$\rho$ if necessary, assume that the closure of~$\phi(B(x_0, \rho))$ is contained in~$B(x_0, \rho)$ and that there is~$\theta$ in~$(0,1)$ such that~$\phi$ contracts distances at least by a factor of~$\theta$.
Note that~$f^{2N}\circ \phi$ is the identity  map on $B(x_0,\rho),$ hence~$\phi$ is increasing on~$B(x_0,\rho)$ and~$f^{2N}$ is increasing on~$\phi(B(x_0,\rho))$.
Since~$f$ is topologically exact, by Lemma~\ref{l:one-sided covering} there is an integer $k'\geq 1$ and a point~$z'$ in $(x_0, x_0+\rho/2)$ such that~$f^{2Nk'}(z') = x_0$ and such that for every $\eps>0$ the set $f^{2Nk'}(B(z',\eps))$ intersects $(x_0,x_0+\rho/2)$. Fix $\eps$ in $(0, |z'-x_0|)$ such that $f^{2Nk'}(B(z',\eps))\subset (x_0,x_0+\rho/2)$. Note that the closure of $B(z',\eps)$ is contained in $(x_0,x_0+\rho/2).$

Let $W$ be the pull-back of $f^{2Nk'}(B(z',\eps))\cap [x_0, x_0+\rho/2)$ by $f^{2Nk'}$ containing $z'$. Since $f^{2Nk'}$, and hence $\varphi^{k'},$ is continuous,   reducing~$\varepsilon$ if necessary, assume that
$ U_0' \= \phi^{k'} \left(f^{2Nk'} (W) \right) $
is disjoint from~$\overline{W}$.
By our choice of~$\phi$, it follows that there is $k_1\geq 0$ such that
$U_1 \= \phi^{k_1} \left( W \right) \subset f^{2Nk'}(W).$
Put
$$ k_0 \= k_1+k'
\text{ and }
U_0\=\phi^{k_1}(U_0'). $$
Then we have
$ k_0\geq 1,
U_0 \cap U_1 = \emptyset,
\text{ and }
U_1 \subset f^{2Nk'} \left( W\right). $
Since~$\phi$ is increasing and contracting and since~$f^{2Nk'} \left(W\right)$ contains~$x_0$, the set
$$ U_0 = \phi^{k_1}(U_0') = \phi^{k_0} \left( f^{2Nk'} \left( \overline{B(z', \varepsilon)} \right) \right) $$
is contained in~$f^{2Nk'} \left( W \right)$. Finally, note that
$ f^{2Nk_0}(U_1)
=
f^{2Nk'} \left( W \right)
=
f^{2Nk_0}(U_0). $
Put
$$ U \=U_0\cup U_1
\text{ and }
\whf \= f^{2Nk_0}|_U. $$

\partn{2}
Put $\hvarphi\= \frac{1}{2Nk_0} S_{2Nk_0}(\varphi)$, for every integer $m\geq 1$ put
$$ \hS_m(\hvarphi)
\=
\hvarphi +\hvarphi\circ \whf+\cdots +\hvarphi\circ \whf^{m-1}, $$
and note that to prove the lemma it suffices to show
\begin{equation}
  \label{e:induced gap pressure}
\limsup_{m\rightarrow \infty}\frac{1}{m}\log \sum_{y\in \whf^{-m}(x_0)} \exp (\hS_m(\hvarphi)(y))
>
\hvarphi(x_0).
\end{equation}
This is equivalent to show that that the radius of convergence of the series
$$ \Xi(s)
\=
\sum_{m = 0}^{\infty} \left( \sum_{z \in \whf^{-m}(x_0)} \exp \left( \hS_m(\hvarphi)(z) \right) \right) s^m $$
is strictly less than~$\exp(- \hvarphi(x_0))$.
To prove this fact, put $\hK\= \bigcap_{i= 0}^\infty \whf^{-i}(U)$ and observe that~$x_0$ is contained in this set.
Consider the itinerary map
$$ \iota: \hK \to \{0,1\}^{\{1, 2, \ldots \}} $$
defined so that for every~$i$ in~$\{1, 2, \ldots \}$ the point~$\whf^i(z)$ is in~$U_{\iota(z)_i}$.
Since~$\whf$ maps each of the sets~$U_0$ and~$U_1$ onto $f^{2Nk'} \left( W \right)$ and both of~$U_0$ and~$U_1$ are contained in this set, for every integer~$k\geq 0$ and every sequence $a_0,a_1, \ldots, a_k$ of elements of~$\{0,1\}$ there is a point of~$\whf^{-(k+1)}(x_0)$ in the set
$$ \hK(a_0 a_1 \cdots a_k)
\=
\left\{z\in \hK: \text{ for every $i$ in $\{0,1,\cdots, k\}$ we have } \iota(z)_i=a_i\right\}. $$
By our choice of~$\phi$ and~$U_0$, there is a constant~$\hC>0$ such that for every integer $k\geq 1$ and every point~$z$ in $\hK(\underbrace{0\cdots0}_k)$, we have
\begin{equation}\label{eq:b1}
\hS_k(\hvarphi)(z) \ge k\hvarphi(x_0) - \hC.
\end{equation}
Taking~$\hC$ larger if necessary, assume that for every point~$z$ in~$U$ we have
 \begin{equation}\label{e:2}
\hvarphi(z)
\ge
\hvarphi(x_0) - \hC.
\end{equation}
It follows that for every integer~$k\geq 0$ and every sequence $a_0, a_1, \cdots, a_k$ of elements of $\{0,1\}$ with $a_0=1$ and every point $x$ in $\hK(a_0a_1\cdots a_k),$ we have
\begin{equation}
\label{eq:b2}
\hS_{k+1}(\hvarphi)(x)\geq (k+1)\hvarphi(x_0)-2(a_0+a_1+\cdots+a_k)\hC.
\end{equation}
In fact, put $\ell\= a_0+\cdots +a_k$, $i_{\ell+1}\=k+1$ and let
$ 0 = i_1 < i_2 < \cdots < i_{\ell} \le k $
be all integers~$i$ in $\{0,1,\cdots,k\}$ such that $a_i=1$.
Then by~(\ref{eq:b1}) and ~(\ref{e:2}) for every $j$ in $\{ 1,\cdots, \ell\}$ we have
$$ \hS_{i_{j+1}-i_j}(\hvarphi)(\whf^{i_j}(x))
\geq
(i_{j+1}-i_j)\hvarphi(x_0)-2\hC. $$
Summing over $j$ in $\{1,2,\cdots, \ell\}$ we obtain~(\ref{eq:b2}).
Thus, if we put
$$ \Phi(s) \= \sum_{k = 1}^{\infty} \exp (k \hvarphi(x_0) - 2 \hC) s^k, $$
then each of the coefficients of
$$ \Upsilon(s) \= \Phi(s) + \Phi(s)^2 + \cdots  $$
is less than or equal to the corresponding coefficient of~$\Xi$, and therefore the radius of convergence of~$\Xi$ is less than or equal to that of~$\Upsilon$.
Since clearly~$\Phi(s) \to \infty$ as~$s \to \exp(- \hvarphi(x_0))^-$, there is~$s_0$ in~$\left( 0, \exp( - \hvarphi(x_0)) \right)$ such that~$\Phi(s_0) \ge 1$.
It follows that the radius of convergence of~$\Upsilon$, and hence that of~$\Xi$, is less than or equal to~$s_0$ and therefore it is strictly less than~$\exp(- \hvarphi(x_0))$.
This completes the proof of~\eqref{e:induced gap pressure} and of the lemma.
\end{proof}

\begin{proof}[Proof of the Key Lemma]
When~$\nu$ is supported on a periodic orbit, the desired inequality follows from Lemma~\ref{l:periodic case}.

Now suppose~$\nu$ is not supported on a periodic orbit. By Proposition~\ref{p:constructing a free IMFS} with $t=0$, there is~$D > 0$, a connected and compact subset~$B_0$ of~$I$, and a free IMFS~$(\phi_k)_{k = 1}^{\infty}$ generated by~$f$ with time sequence~$(m_k)_{k= 1}^\infty$ that is defined on~$B_0$, and such that for every~$k\ge 1$ and every point~$y$ in~$\phi_k(B_0)$ we have
\begin{equation}
\label{e:optimal sum}
S_{m_k}(\varphi)(y)
\ge
m_k\int \varphi \ d\nu -D.
\end{equation}
Since the IMFS~$(\phi_k)_{k = 1}^{\infty}$ is free, there is a point~$x_0$ of~$B_0$ such that for every~$\ul$ and~$\ul'$ in~$\Sigma^*$ such that~$m_{\ul} = m_{\ul'}$, the sets~$\phi_{\ul}(x_0)$ and~$\phi_{\ul'}(x_0)$ are disjoint.
Note that for every integer~$k \ge 1$, every~$\ul = l_1 \cdots l_k$ in~$\Sigma^*$, every~$y_0$ in~$\phi_{\ul}(x_0)$, and every~$j$ in~$\{1, \cdots, k-1 \}$, the point
$$ y_j \= f^{m_{l_1}+m_{l_2}+\cdots +m_{l_{j}}}(y_0) $$
is in~$\phi_{m_{l_{j+1}}}(B_0)$.
Therefore, by~\eqref{e:optimal sum} we have
\begin{multline*}
S_{m_{\ul}}(\varphi)(y_0)
=
S_{m_{l_1}}(\varphi)(y_0) + S_{m_{l_2}}(\varphi)(y_1) + \cdots + S_{m_{l_k}}(\varphi)(y_{k - 1})
\\ \ge
\sum_{i=1}^k \left( m_{l_i}\int \varphi \ d\nu - D \right)
=
m_{\ul}\int \varphi \ d\nu - kD.
\end{multline*}
This shows that for every~$\ul$ in~$\Sigma^*$, and every~$y_0$ in~$\phi_{\ul}(x_0)$ we have
\begin{equation}\label{e:bigsum}
\exp (S_{m_{\ul}}(\varphi)(y_0))\ge \exp
\left(m_{\ul}\int \varphi \ d\nu\right)
\exp(-|\ul|D).
\end{equation}
On the other hand, if for every integer~$n \ge 1$  we put
$$ \Xi_n
\=
\bigcup_{\ul \in \Sigma^*, m_{\ul} = n} \phi_{\ul}(x_0), $$
then the radius of convergence of the series
$$ \Xi(s)
\=
\sum_{n= 1}^\infty\left(\sum_{y\in \Xi_n}\exp (S_n(\varphi)(y))\right)s^n, $$
is given by
$$ R \= \left(\limsup_{n\to \infty} \left(\sum_{y\in \Xi_n}\exp (S_n(\varphi)(y))\right)^{1/n}\right)^{-1}, $$
and satisfies
$$\exp \left(-\limsup_{n\rightarrow \infty}\frac{1}{n}\log \sum_{y\in f^{-n}(x_0)}\exp (S_n(\varphi)(y))\right)\le R. $$
Hence, to complete the proof of Key Lemma it suffices to prove~$R < \exp \left( - \int \varphi \ d\nu \right)$.
Put
$$ \Phi(s)
\=
\sum_{l=1}^\infty \exp(-D)\exp \left(m_{l}\int \varphi \ d\nu\right) s^{m_l}. $$
By inequality~\eqref{e:bigsum} and the fact that $(\phi_k)_{k = 1}^{\infty}$ is free, each of the coefficients of the series
\begin{equation*}
\Upsilon(s)
\=
\sum_{i=1}^\infty\Phi(s)^i
=
\sum_{n = 1}^{\infty} \left(\sum_{\ul \in \Sigma^*, m_{\ul} = n}\exp \left(m_{\ul}\int \varphi \ d\nu\right) \exp(-|\ul| D)\right)s^n,
\end{equation*}
does not exceed the corresponding coefficient of the series~$\Xi$, so the radius of convergence of~$\Xi$ is less than or equal to that of~$\Upsilon$.
Since clearly~$\Phi(s) \to \infty$ as~$s \to \exp\left(-\int \varphi \ d\nu\right)^-$, there exists~$s_0$ in~$\left( 0, \exp\left(-\int \varphi \ d\nu\right) \right)$ such that~$\Phi(s_0)\ge 1$.
This implies that the radius of convergence of~$\Upsilon$, and hence that of~$\Xi$, is less than or equal to~$s_0$, and therefore that~$R \le s_0 < \exp \left(-\int \varphi \ d\nu \right)$. This completes the proof of the Key Lemma.

\end{proof}

\end{document}